\documentclass{amsart}
\usepackage{hyperref}
\usepackage{color}
\usepackage[all,cmtip]{xy}
\usepackage{graphicx}
\usepackage{amssymb}
\usepackage{amsmath}
\usepackage{amsthm}

\usepackage{todonotes}

\newtheorem{theorem}{Theorem}[section]
\newtheorem{corollary}[theorem]{Corollary}
\newtheorem{lemma}[theorem]{Lemma}

\newtheorem{proposition}[theorem]{Proposition}

\newtheorem{question}[theorem]{Question}
\newtheorem{conjecture}[theorem]{Conjecture}
\newtheorem{definition-proposition}[theorem]{Definition-Proposition}
\newtheorem{lemma-notation}[theorem]{Lemma-Notation}

\theoremstyle{definition}
\newtheorem{definition}[theorem]{Definition}
\newtheorem{example}[theorem]{Example}
\newtheorem{remark}[theorem]{Remark}

\newcommand\bP{\mathbb P}

\newcommand{\Z}{\mathbb{Z}}

\newcommand{\C}{\mathbb{C}}
\newcommand{\twopartdef}[4]
{
	\left\{
		\begin{array}{ll}
			#1 & \mbox{if } #2 \\
			#3 & \mbox{if } #4
		\end{array}
	\right.
}

\title{The flex divisor of a K3 surface}
\author{Valery Alexeev}
\email{valery@uga.edu}
\address{Department of Mathematics, University of Georgia, Athens GA
  30602, USA}
\author{Philip Engel}
\email{philip.engel@uga.edu}
\address{Department of Mathematics, University of Georgia, Athens GA
  30602, USA}

\begin{document}

\begin{abstract}
The {\it flex divisor} $R_{\rm flex}$ of a primitively polarized K3 surface $(X,L)$ 
is, generically, the set of all points $x\in X$ for which there exists
a pencil $V\subset |L|$ whose base locus is $\{x\}$. We show
that if $L^2=2d$ then $R_{\rm flex}\in |n_dL|$ with
$$n_d= \frac{(2d)!(2d+1)!}{d!^2(d+1)!^2} =(2d+1)C(d)^2,$$
where $C(d)$ is the Catalan number.
We also show that there is a well-defined notion of flex divisor over
the whole moduli space $F_{2d}$ of polarized K3 surfaces.
\end{abstract}
\maketitle

\section{Introduction}

Let $(X,L)$ be a primitively polarized K3 surface of degree $2d$. Recent work of the authors
on compactification of the moduli space $F_{2d}$ of such
surfaces
has highlighted the importance of a
{\it canonical choice of polarizing divisor}: An algebraically varying
choice of divisor $R\in |nL|$ on the generic polarized K3 surface.
If this choice of divisor extends over all of $F_{2d}$ then it gives rise to
a modular compactification $$F_{2d}\hookrightarrow \overline{F}_{2d}^R.$$ The compactification
is constructed by taking the closure of the space of pairs $(X,\epsilon R)$ in the moduli space
of stable slc pairs, for some small $\epsilon>0$.

By the main theorem of \cite{alexeev2021compact}, the normalization of 
$\overline{F}_{2d}^R$ is semitoroidal whenever $R$ satisfies a property
called {\it recognizability}. Thus, the search for modular toroidal compactifications
of $F_{2d}$ is intimately related to finding canonical choices
of polarizing divisor, and verifying that those choices
are recognizable.

One infinite series of divisors, ranging over all degrees $2d$, is the {\it rational curve divisor}.
On a generic K3 surface $(X,L)$ it can be defined as
$$R_{\rm rc}:=\sum_{\substack{C\in |L| \\ \textrm{rational}}}C$$
and was proven to be recognizable in \cite{alexeev2021compact}.
By the famous Yau-Zaslow formula \cite{yau1996bps, beauville1999counting},
$R_{\rm rc}\in |n_dL|$ where $n_d = [q^{d+1}]\,\prod_{n\geq 1} (1-q^n)^{-24}$.

Claire Voisin suggested to the authors a second series of divisors, which we call
here the {\it flex divisor} $R_{\rm flex}$. It was first considered by Welters
\cite{welters1981abel} for a quartic K3 surface, who called it the
curve of hyperflexes.
On the generic $(X,L)$ it is defined as the
set of all points $x\in X$ for which there exists a pencil $V\in |L|$ whose set-thereotic
base locus is $\{x\}$.
When $|L|$ defines an embedding $X\hookrightarrow \mathbb{P}^g$ with $g=d+1$,
which is the case for generic $(X,L)\in F_{2d}$ when $d\geq 2$, the flex divisor can be concretely thought
of as the set of points $x\in X\subset \mathbb{P}^g$ for which there exists a {\it flex space}: A codimension
$2$ linear subspace of $\mathbb{P}^g$ intersecting $X$ at only the point~$x$.

Our first result hints towards a positive answer on the question of whether $R_{\rm flex}$ is recognizable.
Concretely, we show:

\begin{theorem}\label{flex-exists} There is a canonical choice of divisor $R_{\rm flex}$ 
varying algebraically over all of $F_{2d}$ and giving the flex divisor on the
generic K3 surface $(X,L)$. \end{theorem}

The theorem is not obvious, because it is not clear if the
flex points form a subvariety of $X$ of the expected dimension, which is one.
Additionally, the flex divisor may have multiple components and
 one must determine their multiplicities. Perhaps most importantly,
sometimes points in $R_{\rm flex}$ as in Theorem \ref{flex-exists}
are not flex under the naive definition!
This occurs on a quartic surface containing a line---the points on the line are not
flex according to the naive definition because the relevant pencil $V$ contains the whole line as a base
curve. But the line appears as a component of the flex divisor, see Example \ref{deg4ex}.

The flex divisor is notably an example of a {\it constant cycle curve}
\cite{huybrechts2013curves}:
One whose points all have the same
class in the Chow group of zero-cycles ${\rm CH}_0(X)$. The method of proof of
Theorem \ref{flex-exists} suggests strongly:

\begin{conjecture} Let $R$ be a canonical choice of polarizing divisor for $F_{2d}$. If $R$
is a constant cycle curve, then it is recognizable. \end{conjecture}

A resolution of this conjecture would unify various results about recognizable divisors, such
as \cite{alexeev2019stable}, \cite{alexeev2020compactifications}, \cite{alexeev2021compact},
and \cite{alexeev2021nonsymplectic}.

Our second result is an analogue of the Yau-Zaslow formula. That is, we
determine in what multiple of the polarization the flex divisor lives,
generalizing known results in the cases $d=1,2$.

\begin{theorem} Let $(X,L)$ be a K3 surface of degree $2d$. Then $R_{\rm flex}\in |n_dL|$ with
$$n_d= \frac{(2d)!(2d+1)!}{d!^2(d+1)!^2} =(2d+1)C(d)^2,$$
where $C(d)$ is the Catalan number.
\end{theorem}

\vspace{-5pt}

\begin{table}[h]
\begin{tabular}{c|ccccccccc}
$d$ & $1$ & $2$ & $3$ & $4$ & $5$ & $6$ & $7$ & $8$ & $9$   \\
\hline
$n_d$ & $3$ & $20$ & $175$ & $1764$ &
$19404$ & $226512$ & $2760615$ & $34763300$ & $449141836$ \\
\end{tabular}\vspace{3pt}
\caption{Flex divisor classes}
\label{tab-values}
\end{table}

\vspace{-5pt}

Table \ref{tab-values} tabulates the first nine values of $n_d$. The value $n_1=3$
is well-known, see Example \ref{deg2ex},
while the value $n_2=20$ has been computed by various authors
\cite[Prop.~8.8]{huybrechts2013curves},
\cite[Cor.~2.4.6]{witaszek2014geometry}.

The summary of the paper is as follows: Section \ref{sec:flex-defined} shows that the flex divisor
is well-defined and extends to a divisor over all of $F_{2d}$ and Section \ref{sec:flex-degree} computes
the multiple $n_d$ of the primitive polarization in which the flex divisor lives, using
 intersection theory on the Hilbert scheme $X^{[2d]}$ of the K3 surface.

\section{Well-definedness}
\label{sec:flex-defined}

\begin{definition} We say $(X,L)$ is a {\it polarized K3 surface of degree $2d$} if $X$
is a K3 surface with ADE singularities, and $L\to X$ is a primitive, ample line bundle
satisfying $L^2=2d$. 
\end{definition}

Let $F_{2d}$ denote the moduli stack of polarized K3 surfaces over $\C$.
It is a smooth, irreducible, Deligne-Mumford stack of dimension $19$.

\begin{definition} A point $x\in X$ is {\it flex} if there exists a pencil $V\subset |L|$ whose
base locus is the singleton $\{x\}$. \end{definition}

Let $L_{K3}$ denote the unique
even, unimodular lattice of signature $(3,19)$ and fix a primitive vector $v\in L_{K3}$
 of norm $v^2=2d$. Define
\begin{align*} \mathbb{D}&:=\mathbb{P}\{x\in v^\perp\otimes \C\,\big{|}\,x\cdot x=0,\,x\cdot \overline{x}>0\} \textrm{ and} \\ 
\Gamma&:=\{\gamma\in O(L_{K3})\,\big{|}\,\gamma(v)=v\}.\end{align*}
By the Torelli theorem, the coarse space of $F_{2d}$ 
is the arithmetic quotient $\mathbb{D}/\Gamma$. 

\begin{definition} A {\it Heegner divisor} in $\mathbb{D}/\Gamma$ is the image of a hyperplane section
$w^\perp\cap \mathbb{D}$ for some vector $w\in L_{K3}\setminus \Z v$. \end{definition}

\begin{proposition} Let $S\subset F_{2d}$ denote the
polarized K3 surfaces $(X,L)$ for which there exists a pencil
$V\subset |L|$ containing a base curve. Then $S$ is contained in a 
finite union of Heegner divisors.
\end{proposition}

\begin{proof} The condition that $|L|$ contain a pencil with a base curve
is an algebraic condition, which is easily seen to be closed on $F_{2d}$. 

Let $C$ be the base curve of a pencil $V\subset |L|$.
Fix $H\in V$ and note $H = C+D$ for some non-empty effective divisor $D$.
Since $L$ is ample, we have $0<L\cdot C<2d$.
Thus $[C]\notin \Z L$. By the primitivity of $L$, the rank of ${\rm Pic}(X)$
is least $2$. Hence any point of $S$ lies in some Heegner divisor. 
Since $S$ is algebraic, we conclude that
$S$ is contained in a finite union of Heegner divisors.
\end{proof}

\begin{lemma}\label{not-all-flex} Let $(X,L)$ be a polarized K3 surface.
The flex points $\{x\in X\,\big{|}\,x\textrm{ flex}\}$ form a constructible
subset of $X$ of dimension at most $1$.  \end{lemma}

\begin{proof} Constructibility is elementary. To show the second statement,
it suffices to make the following observation: Any flex point $x\in X$ lies in the
Beauville-Voisin class $[x]=c_X\in {\rm CH}_0(X)$, defined in \cite{beauville2004chow}
as the class of any point on a rational curve in $X$. This follows because:
\begin{enumerate} 
\item $2d[x]=H_1\cdot H_2$
for hyperplane sections $H_1, H_2$ spanning the pencil $V$,
\item the intersection of two curves is some multiple of $c_X$ \cite[Thm.~1]{beauville2004chow}, and 
\item ${\rm CH}_0(X)$ is torsion-free \cite{rojtman1980torsion}. 
\end{enumerate}
If a Zariski-open subset of points of $X$ were flex, we would have that ${\rm CH}_0(X)=\Z$,
contradicting Mumford's theorem \cite{mumford1969rational}
on the uncountability of the Chow group. So
the constructible set of flex points has dimension at most $1$. \end{proof}

Given a smooth surface $X$, denote by $X^{[k]}$ the Hilbert
scheme of $k$ points on $X$. Let $F_{2d}^{\rm sing}$ denote the substack of $F_{2d}$
parameterizing singular ADE K3 surfaces, which is also a finite
union of Heegner divisors.

\begin{definition} Define the Zariski open subset 
$T:=F_{2d}\setminus (S\cup F_{2d}^{\rm sing})$.
 We assume for the remainder of the text
 that $(X,L)\in T$, unless otherwise stated.\end{definition}

Let $G:=\textrm{Gr}(g-1,g+1)$
be the Grassmannian of codimension $2$ linear spaces in
$H^0(X, L)^*$,
or equivalently pencils in $|L|$. Consider the map
\begin{displaymath}
  i\colon G\to X^{[2d]},\qquad V\mapsto \mathbb PV\cap X
\end{displaymath}
sending a codimension
$2$ linear space to its scheme-theoretic intersection with $X$, or equivalently sending
a pencil to its scheme-theoretic basic locus.

\begin{proposition}\label{closed-immersion} The mapping $i\colon G\to X^{[2d]}$ is a closed immersion. \end{proposition}

\begin{proof} First, we show that $i$ is a set-theoretic embedding. Suppose,
for the sake of contradiction, that two pencils
$\mathbb PV_1\cap X = \mathbb PV_2\cap X$ intersect at the same length $2d$ subscheme $Z\subset X$.
Consider the set of codimension
$2$ linear spaces $$\mathbb{P}:=\{V\in G\,\big{|}\,V\supset V_1\cap V_2\}.$$ We necessarily have 
that $\bP V\cap X=Z$ for all $V\in \mathbb{P}$ because $Z\subset\bP V_1\cap\bP V_2\cap X$.
Hence $i(\mathbb{P})$ consists
of a single point. Since $\mathbb{P}$ contains a curve, we conclude that $i$ contracts a curve. But
any morphism from a Grassmannian to a projective variety contracting a curve must be constant.
So $i$ is constant, which is absurd.

Next, we show that the differential $di$ is injective. Recall that the tangent space
$T_VG={\rm Hom}(V,\,H^0(X,L)^*/V)$ whereas
$T_{[Z]}X^{[2d]}={\rm Hom}(I_Z /I_Z^2,\,\mathcal{O}_Z)$. We may write
$\mathbb PV=\{x\in \mathbb{P}^g\,\big{|}\, s_1(x)=s_2(x)=0\}$
for two sections $s_1,s_2\in H^0(X,L)$. 
A tangent vector $T_VG$ can be represented as the
vanishing locus of $(s_1+\epsilon t_1, s_2+\epsilon t_2)$,
where $t_1,t_2\in H^0(X,L)/\C s_1\oplus \C s_2$. 
Then $di$ maps it to $(s_1\mapsto t_1|_Z, s_2\mapsto t_2|_Z)$,
which uniquely determines an element of
${\rm Hom}(I_Z /I_Z^2,\,\mathcal{O}_Z)$ because
$I_Z=( s_1,s_2)$.
  
Supposing some nonzero $\phi\in T_VG$ satisfies $di(\phi)=0$,
at least one of $t_1, t_2\in H^0(X,L)/\C s_1\oplus \C s_2$ is
nonzero and satisfies $t_i \big{|}_Z=0$.
So $Z$ is contained in the codimension $3$
linear space $\{x\in \mathbb{P}^g\,\big{|}\,s_1(x)=s_2(x)=t_i(x)=0\}$.
But then the argument of the first paragraph applies to show
$i$ is constant. Contradiction.
 \end{proof}
 
Consider the Hilbert-Chow morphism $HC\colon X^{[2d]}\to X^{(2d)}$. Let 
$\Delta\subset X^{(2d)}$ be the small diagonal of effective zero cycles of the form $
2d[x]$ for some $x\in X$. Define a subscheme $P_{2d}\subset X^{[2d]}$
as the scheme-theoretic fiber $P_{2d}:=HC^{-1}(\Delta).$
Let $${\rm supp}\colon P_{2d}\to X$$ be the support morphism, sending
a scheme to the point at which it is supported. Finally,
let $i(G)\subset X^{[2d]}$ denote the image of the Grassannian
$G={\rm Gr}(g-1,g+1)$ under the morphism $i$, endowed with
its natural structure of a reduced, smooth subscheme. Finally,
we may now describe the flex divisor, at least generically.

\begin{definition} 
The {\it flex divisor} on a K3 surface $(X,L)\in T$
is the algebraic cycle
$$R_{\rm flex}:={\rm supp}_*[P_{2d}\cap i(G)].$$ Here
${\rm supp}_*$ denotes the proper pushforward of algebraic
cycles, and the brackets $[\cdot]$ denote the 
effective algebraic cycle underlying a subscheme.
 \end{definition}
 
 Note that the cycle class is being taken in $P_{2d}$ to make ${\rm supp}_*$
 sensical.
 
\begin{lemma}\label{equidim-flex} The subschemes $P_{2d}$ and $i(G)$ intersect
properly in $X^{[2d]}$, i.e. their intersection has pure dimension $1$. Furthermore,
$[P_{2d}]\cdot [i(G)]=[P_{2d}\cap i(G)]_{X^{[2d]}}.$ 
\end{lemma}

\begin{proof} We have that $i(G)\subset X^{[2d]}$ is a smooth subscheme
of dimension $2d$. By a result of Haiman \cite[Prop.~2.10]{haiman1998t},
$P_{2d}\subset  X^{[2d]}$
is a reduced and irreducible 
Cohen-Macaulay scheme of dimension $2d+1$.
Hence, each component of their intersection has dimension
at least $1$. We claim additionally that each
component has dimension at most $1$.
Note ${\rm supp}(P_{2d}\cap i(G))\subsetneq X$ by Lemma \ref{not-all-flex}. 

The restriction of ${\rm supp}$ to $P_{2d}\cap i(G)$ contracts no curves
becuase no flex point $x\in X$ has a curve-worth of flex spaces:
If $x\in X$ supported a curve-worth of flex spaces, the morphism $HC\circ i :G\to X^{(2d)}$ would
contract a curve and hence, as before, $G$ would collapse
to a point in $X^{(2d)}$. This is absurd. So ${\rm supp}\big{|}_{P_{2d}\cap i(G)}$
is finite onto its image in $X$, which has dimension at most $1$.

Hence, each component of $P_{2d}\cap i(G)$
has dimension exactly $1$, that is,
$P_{2d}$ and $i(G)$ intersect properly.
Since $i(G)$ is smooth and $P_{2d}$ is Cohen-Macaulay,
\cite[Prop.~7.1]{fulton2016intersection}
gives the second statement.
\end{proof}

\begin{remark}\label{points-flex} The proof of Lemma \ref{equidim-flex} implies that every
component of the scheme $P_{2d}\cap i(G)$
contributes nontrivially to $R_{\rm flex}$. Hence, for
$(X,L)\in T$, $R_{\rm flex}$ is, {\it as a set}, exactly
the set of flex points. \end{remark}

\begin{proposition} Let
$u\colon \mathfrak{X}\to T$ be the restriction of the universal family of polarized K3 surfaces.
Then the flex divisors $\mathfrak{R}_{\rm flex}\subset \mathfrak{X}$ form a flat
subfamily of curves, specializing to $R_{\rm flex}$ on any fiber $X=\mathfrak{X}_t$.
 \end{proposition}
 
 \begin{proof} It suffices to relativize the construction of $R_{\rm flex}$
 and check flatness of the resulting family of algebraic cycles.
 
  Let $\mathfrak{G}$ be the relative Grassmannian of codimension $2$
 linear subspaces of $\mathbb{P}(u_*\mathfrak{L})^*$ where $\mathfrak{L}\to \mathfrak{X}$
 is the universal polarization. Let $\mathfrak{X}^{[2d]}$ be the relative Hilbert scheme
 of $2d$ points, and let $\mathfrak{P}_{2d}$ be the subfamily of the relative Hilbert
 scheme consisting of schemes supported at a single point of the fiber and $\mathfrak{i}$
 the relative inclusion $\mathfrak{G}\hookrightarrow \mathfrak{X}^{[2d]}$. Let
 $\mathfrak{supp}:\mathfrak{P}_{2d}\to \mathfrak{X}$ be the relative support morphism.
 Consider the algebraic cycle
$$\mathfrak{R}_{\rm flex}:=
\mathfrak{supp}_*[\mathfrak{P}_{2d}\cap \mathfrak{i}(\mathfrak{G})]\subset \mathfrak{X}.$$
This cycle is a divisor in the smooth DM stack
$\mathfrak{X}$. Any fiber
 $X=\mathfrak{X}_t\hookrightarrow \mathfrak{X}$ intersects
 $\mathfrak{R}_{\rm flex}$ properly by Lemma \ref{equidim-flex}. Hence $\mathfrak{R}_{\rm flex}$
 forms a flat family of divisors in $\mathfrak{X}$.
 
 It remains to show that $\mathfrak{R}_{\rm flex}$
 specializes to $R_{\rm flex}$ as defined above on a fiber $X=\mathfrak{X}_t$. 
 The pushforward $\mathfrak{supp}_*$ of algebraic cycles and
 the cycle class map $[\cdot]$ commute
 with taking fibers over $t$ because the fibers $X^{[2d]}\hookrightarrow \mathfrak{X}^{[2d]}$
 are smoothly immersed and properly intersecting the cycles
 $\mathfrak{G}$ and $\mathfrak{P}_{2d}$.
 Hence,
 $(\mathfrak{R}_{\rm flex})_t=R_{\rm flex}$. \end{proof}
  
\begin{question} For a sufficiently generic $(X,L)\in F_{2d}$ is $R_{\rm flex}$
an irreducible divisor? What is its geometric genus, generically? \end{question}

\begin{remark} Based off \cite{welters1981abel},
Huybrechts \cite[Prop.~8.8]{huybrechts2013curves}
shows that when $L^2=4$, $R_{\rm flex}\in |20L|$ is generically
irreducible of geometric genus $201$. 
Strangely, this is the genus of a smooth
element of $|10L|$.
This is not an
error: $R_{\rm flex}$ is generically singular
for a quartic surface. \end{remark} 

We recall now the notion of a constant cycle curve:

\begin{definition} Let $X$ be a smooth K3 surface, and let $R\subset X$ be a curve.
 We say that $R$ is a {\it constant cycle curve} if every point $p\in R$ represents
 the same class in ${\rm CH}_0(X)$. This definition extends to curves $R\subset X$ in
 an ADE K3 surface by taking the inverse image of $R$
 in the minimal resolution of $X$. \end{definition}
 
 It is known that if $R$ is constant cycle, then $[p]=c_X\in {\rm
   CH}_0(X)$ for all $p\in R$.
 
 \begin{lemma} For $(X,L)\in T$,
 the divisor $R_{\rm flex}$ is a constant cycle curve. \end{lemma}
 
 \begin{proof} This follows immediately from Remark \ref{points-flex} and items
 (1), (2), (3) in the proof of Lemma \ref{not-all-flex}.
 \end{proof}
 
\begin{lemma}\label{specialize} Let $\mathcal{X}\to (C,0)$ be a family of polarized K3 surfaces
and let $\mathcal{R}\subset \mathcal{X}$ be a flat family of curves over $C$. Suppose
that $\mathcal{R}_t$ is a constant cycle curve for all $t\neq 0$. Then
$\mathcal{R}_0\subset \mathcal{X}_0$ is also a constant cycle curve. \end{lemma}

\begin{proof} Replacing $\mathcal{X}$ with a finite base change,
there is a simultaneous resolution of singularities which is the minimal
resolution on any fiber. So we may assume $\mathcal{X}\to (C,0)$
is smooth. Any two points $p,q\in \mathcal{R}_0$ can
be realized as specializations of points over a finite extension of $\C(C)$.
The lemma follows because the specializations of 
rationally equivalent cycles are rationally equivalent
\cite[Cor.~20.3]{fulton2016intersection}. \end{proof}

\begin{theorem} Let $u\colon \mathfrak{X}\to F_{2d}$ be the universal K3 surface,
$T\subset F_{2d}$ a Zariski open subset,
and let $\mathfrak{R}^*\subset \mathfrak{X}^*:=\mathfrak{X}\big{|}_T$
be a flat family of divisors, which is a constant cycle curve $R=\mathfrak{R}_t$
on every fiber $X=\mathfrak{X}_t$. Then $\mathfrak{R}^*$ extends to a flat family of divisors
$\mathfrak{R}$ over the universal K3 surface $\mathfrak{X}\to F_{2d}$. \end{theorem}

\begin{proof} Let $\mathfrak{L}$ be an extension of
$\mathcal{O}_{\mathfrak{X}^*}(\mathfrak{R}^*)$ to $\mathfrak{X}$ and define
the projective bundle $\mathbb{P}(u_*\mathfrak{L})\to F_{2d}$. By assumption,
we have a section of $\mathbb{P}(u_*\mathfrak{L})$ over the open subset $T$
defined by $\mathfrak{R}^*$.
Let $0\in F_{2d}\setminus T$. Given any arc $(C,0)\subset F_{2d}$ with
$C\setminus \{0\}\subset T$, there is a unique flat family of curves
$\mathcal{R}\subset \mathfrak{X}\big{|}_C$ extending
$\mathfrak{R}^*\big{|}_{C\setminus \{0\}}$.

By Lemma \ref{specialize},
the central fiber $\mathcal{R}_0$ is constant cycle. 
As noted in \cite[Sec.~2.3]{huybrechts2013curves},
Mumford's theorem \cite{mumford1969rational}
implies constant cycle curves are rigid.
So the flat limit $\mathcal{R}_0$ doesn't deform as the arc
$(C,0)$ deforms. Since $F_{2d}$ is smooth, in particular normal,
we conclude by a well-known argument
\cite[Lem.~3.16]{alexeev2019stable}
that the section of $\mathbb{P}(u_*\mathfrak{L})$ over $T$
extends, as a morphism, over $0$.
 The result follows.\end{proof}

\begin{corollary} $\mathfrak{R}_{\rm flex}$ extends to a flat family of divisors in
the universal K3 surface over $F_{2d}$. \end{corollary}

\section{Degree of the Flex Divisor}
\label{sec:flex-degree}

In this section, we compute the degree of the flex divisor. We follow
\cite{ellingsrud2000hilbert}
as a general reference on the cohomology of Hilbert schemes.
%
%


\begin{definition}\label{nakajima} Let $n>0$ be a positive integer and let $\alpha\in H^*(X)$ be a
cohomology class of pure degree. Define $$\mathbb{L}:=\bigoplus_{m,\,k\geq 0} H^m(X^{[k]}).$$
The {\it Nakajima (raising) operator} $q_{-n}(\alpha):\mathbb{L}\to \mathbb{L}$ is defined by
the following correspondence: Let $a\geq 0$ and define $b:=a+n$. Let $X^{[a,b]}$ 
be the incidence correspondence of nested pairs of zero-dimensional 
subschemes $Z_1\subset Z_2\subset X$ for which ${\rm len}\,Z_1=a$ and ${\rm len}\,Z_2=b$,
and let $\pi_a$ and $\pi_b$ be the projections to $X^{[a]}$ and $X^{[b]}$. Let $S$ be the 
residual support morphism $X^{[a,b]}\to X^{(n)}$ sending
$$S\colon (Z_1, Z_2)\mapsto {\rm supp}(Z_2)-{\rm supp}(Z_1)$$ and let $W_{a,b}\subset S^{-1}(\Delta)$
be the irreducible component of $S^{-1}(\Delta)$ which is the Zariski closure of the $Z_1\subset Z_2$
for which ${\rm supp}(Z_1)$
and ${\rm supp}(Z_2)-{\rm supp}(Z_1)$ are disjoint. Let
$s:W_{a,b}\to \Delta\cong X$ denote the restriction of $S$
and let $t:W_{a,b}\to X^{[a,b]}$ be the inclusion.
Then for any $c\in H^r(X^{[a]})$ we define
$$q_{-n}(\alpha) (c):=(\pi_b)_*(\pi_a^*c \cdot t_*s^*\alpha)\in H^{r+2n-2+\deg \alpha}(X^{[b]}).$$
By definition, we declare $H^*(X^{[0]})=\C{\bf 1}$ where ${\bf 1}$ is called
the {\it vacuum element}.
\end{definition}

The bidegree of the operator $q_{-n}(\alpha)$ is $(2n-2+\deg\alpha,n)$, where
the first degree is cohomological degree, and the second is number of points.

\begin{remark} Definition \ref{nakajima} can be
intuitively rephrased as follows: The operator $q_{-n}(\alpha)$ takes a family of
subschemes of length $a$ (i.e. a cycle in $X^{[a]}$)
and tacks on a subscheme of length $n$ supported at a single point lying
on the cycle $\alpha$. \end{remark}

\begin{theorem}[Nakajima \cite{nakajima1997heisenberg},
Grojnowski \cite{grojnowski1996instantons}] Let $\{e_i\}_{i=1}^{24}$ be a basis of $H^*(X)$.
Then $q_{-n_1}(e_{i_1})\cdots q_{-n_k}(e_{i_k}){\bf 1}$
(up to reordering operators) are a basis of~$\,\mathbb{L}$.
 \end{theorem}
 
 More precisely, these Nakajima operators extend to 
 an action of the Heisenberg algebra of
 $H^*(X)$ on $\mathbb{L}$, which becomes identified
with the bosonic Fock space.

\begin{remark} 
It follows directly from the definition of the Nakajima operators
that $[P_{2d}]=q_{-2d}(1){\bf 1}$. Similarly, the schemes supported
on a single point of a hyperplane section $H\subset X$ 
have class $q_{-2d}(h){\bf 1}$, with $[H]=h\in H^2(X)$. \end{remark}

%
%
%

\begin{lemma}\label{first-formula} The degree of the flex divisor is
${\rm deg}(i^*q_{-2d}(h){\bf 1}).$ \end{lemma} 

\begin{proof} By push-pull formula,
\begin{align*} {\rm deg}(R_{\rm flex})=R_{\rm flex}\cdot_X H :&= {\rm supp}_*[P_{2d}\cap i(G)]\cdot_X H
= [P_{2d}\cap i(G)]\cdot_{P_{2d}} {\rm supp}^*H \\ &=i(G)\cdot_{X^{[2d]}}  q_{-2d}(h){\bf 1} = {\rm deg}(i^*q_{-2d}(h){\bf 1}).\end{align*} \end{proof}

Let $\sigma_i\textrm{ for }i=1,2$ denote the
 Schubert classes in $H^{2i}(G)$ consisting of linear spaces
 meeting a line and a point in $\mathbb{P}^g$, respectively.

\begin{proposition}\label{int-lemma} The degree of the flex divisor is
  $\sigma_1\cdot i^*q_{-2d}(1){\bf 1}.$ \end{proposition}

\begin{proof} The first step is to verify the intersection
product $$q_{-1}(h)q_{-1}(1)^{2d-1}{\bf 1}\cdot q_{-2d}(1){\bf 1}   =2dq_{-2d}(h){\bf 1}$$
on $X^{[2d]}$ and the second step is to verify that
$i^*(q_{-1}(h)q_{-1}(1)^{2d-1}{\bf 1})=2d\sigma_1$. Then we can
apply Lemma \ref{first-formula}. The first step is
set-theoretically clear; the intersection multiplicity $2d$
follows quickly from the description of the ring structure on
$H^*(X^{[2d]})$ due to Lehn and Sorger \cite[Thm.~1.1 and Prop.~2.13]{lehn2003cup}.

To verify the second step, note that 
$q_{-1}(h)q_{-1}(1)^{2d-1}{\bf 1}$ is represented by
the divisor $D_H\subset X^{[2d]}$
of schemes whose support intersects $H\subset X$.
Thus $[i^{-1}(D_H)]$ represents $i^*(q_{-1}(h)q_{-1}(1)^{2d-1}{\bf 1})$.
But $i^{-1}(D_H)$ is simply the locus of codimension $2$
linear spaces passing through some point of $H$.
Since $[H]_{\mathbb{P}^g}=2d\ell$ where $\ell$
is the line class in $\mathbb{P}^g$, we conclude that
$[i^{-1}(D_H)]=2d\sigma_1$.
\end{proof}

Let $\mathcal{Z}\subset X^{[2d]}\times X$ denote the universal subscheme of length $2d$.
Let $\mathcal{Z}_G\subset G\times X$
denote the restriction of this subscheme to $G$ (along the inclusion $i$). Let $\pi_{X^{[2d]}}$ and $\pi$
denote the projections from $X^{[2d]}\times X$ and $G\times X$ to the first factor, respectively.
The {\it tautological bundle} $\mathcal{O}^{[2d]}\to X^{[2d]}$
is the pushforward $(\pi_{X^{[2d]}})_*\mathcal{O}_{\mathcal{Z}}$ and is a vector bundle
of rank $2d$ on $X^{[2d]}$. Let
$$\mathcal{O}^{[2d]}_G:=i^*\mathcal{O}^{[2d]}=\pi_*\mathcal{O}_{\mathcal{Z}_G}$$
denote the restriction of  this vector bundle to the Grassmannian $G$.

\begin{proposition}\label{int-lemma2}
We have $i^*q_{-2d}(1)=-2dc_{2d-1}(\mathcal{O}_G^{[2d]}).$ \end{proposition}

\begin{proof} Applying \cite[Thm.~12.4]{ellingsrud2000hilbert} 
to the line bundle $\mathcal{O}$ gives the formula
$$\sum_n c(\mathcal{O}^{[n]}) = {\rm exp}
\left(\textstyle\sum_{m\geq 1} \frac{(-1)^{m-1}}{m}q_{-m}(c(\mathcal{O}))\right).$$
Note $c(\mathcal{O})=1$ and that $q_{-m}(1)$ has bidegree $(2m-2, m)$.
So the only term on the right-hand side landing in $H^{2n-2}(X^{[n]})$
is $(-1)^{n-1}n^{-1}q_{-n}(1)$. We conclude \begin{align*}
i^*q_{-2d}(1)&=-2di^*c_{2d-1}(\mathcal{O}^{[2d]})=-2dc_{2d-1}(\mathcal{O}_G^{[2d]}) \end{align*}
which follows via commutativity of taking Chern classes with pullback.  \end{proof}

Let $Q$ denote the rank $2$ universal quotient bundle on $G$.
To compute the Chern class $c_{2d-1}(\mathcal{O}_G^{[2d]})$ we make
use of the following exact sequence:

\begin{proposition} \label{koszul}
There is a resolution of $\mathcal{O}_{\mathcal{Z}_G}$ by vector bundles on $G\times X$:
$$0\to \det(Q^*)\boxtimes (-2L) \to Q^*\boxtimes (-L)\to \mathcal{O} \to \mathcal{O}_{\mathcal{Z}_G} \to 0.$$
\end{proposition}

\begin{proof} This exact sequence is simply the 
global version of the Koszul resolution of $\mathcal{O}_{X\cap \mathbb{P}V}$ 
where $\mathbb{P}V=\{x\in \mathbb{P}^g\,\big{|}\, s_1(x)=s_2(x)=0\}$
is a codimension $2$ linear space:
$$0\to (s_1s_2)\to (s_1)\oplus (s_2)\to \mathcal{O}_X \to \mathcal{O}_{X\cap \mathbb{P}V}\to 0.$$
On a given fiber of $\pi$ the restrictions of $\det(Q^*)\boxtimes (-2L)$
and $Q^*\boxtimes (-L)$ are $(s_1s_2)$ and $(s_1)\oplus (s_2)$ respectively,
because $Q^*=\C s_1\oplus \C s_2$.
\end{proof}

Let $r_1$ and $r_2$ denote the Chern roots of $Q$. 

\begin{proposition} ${\rm ch}(\mathcal{O}_G^{[2d]}) = 2-(d+2)e^{-r_1}-(d+2)e^{-r_2}+(4d+2)e^{-r_1-r_2}$. \end{proposition}

\begin{proof} Consider the (derived) pushforward $R\pi_*$
of the exact sequence of Proposition \ref{koszul}. Computing the
derived pushforward sheaves of each term gives
$$R^i\pi_* \mathcal{O}_{\mathcal{Z}_G}=\twopartdef{\mathcal{O}_G^{[2d]}}{i=0}{0}{i>0} \hspace{20pt}
R^i\pi_* \mathcal{O}=\twopartdef{\mathcal{O}}{i=0,2}{0}{i=1}$$
$$R^i\pi_*(Q^*\boxtimes (-L))=\twopartdef{0}{i=0,1}{Q^*\otimes H^0(X,L)^*}{i=2}$$
$$R^i\pi_*(\det(Q^*)\boxtimes (-2L))=\twopartdef{0}{i=0,1}{\det(Q^*)\otimes H^0(X,2L)^*}{i=2.}$$
The first equation follows from the definition of $\mathcal{O}_G^{[2d]}$ and that
$\mathcal{Z}_G$ is finite over $G$, and the last three equations all follows from relative Serre
duality applied to $\pi$. From these computations, and the fact that 
$h^0(X,L)=d+2$ and $h^0(X,2L)=4d+2$, we get the following equality in the $K$-group of $G$:
$$[\mathcal{O}_G^{[2d]}]-2[\mathcal{O}]+(d+2)[Q^*]-(4d+2)[\det(Q^*)]=0.$$
Since the Chern character ${\rm ch}$ is a homomorphism from $K$-theory to cohomology, the proposition
follows from the equalities
${\rm ch}(\mathcal{O})=1$, ${\rm ch}(Q^*)=e^{-r_1}+e^{-r_2}$, ${\rm ch}(\det(Q^*))=e^{-r_1-r_2}$. \end{proof}

\begin{corollary}\label{tot-chern} The total Chern character of $\mathcal{O}_G^{[2d]}$ is $$c(\mathcal{O}_G^{[2d]})=\frac{(1-r_1-r_2)^{4d+2}}{(1-r_1)^{d+2}(1-r_2)^{d+2}} = \frac{(1-\sigma_1)^{4d+2}}{(1-\sigma_1+\sigma_2)^{d+2}} .$$ \end{corollary}

\begin{proof} Since $\mathcal{O}_G^{[2d]}$ is a vector bundle, we can compute
the total Chern character using the splitting principle and the set of ``virtual Chern roots"
$$\{\underbrace{-r_1-r_2}_{4d+2}, 0,0\}-\{\underbrace{-r_1}_{d+2},\underbrace{-r_2}_{d+2}\}.$$
The theorem then follows from the equalities $r_1+r_2=\sigma_1$ and $r_1r_2=\sigma_2$.
\end{proof}

\begin{remark} Let
$X\subset \mathbb{P}^g$ be Cohen-Macaulary
of degree $d$ and codimension $r$.
If $X$ intersects any
$r$-plane in $\mathbb{P}^g$
properly, there is a map ${\rm Gr}(r+1,g+1)\to X^{[d]}$
which sends an $r$-plane to its intersection with $X$. 
There is a rank $d$ tautological vector
bundle $\mathcal{O}^{[d]}\to X^{[d]}$ and 
the Chern classes of its pullback to ${\rm Gr}(r+1,g+1)$
can be computed in the same manner 
as above, via the Koszul resolution.
  \end{remark}

\begin{theorem}\label{final} Let $X\subset \mathbb{P}^g$ be a smooth K3 surface
embedded by a primitive ample line bundle $L$ of square $L^2=2d=2g-2$,
for which no pencil in $|L|$ has a base curve.
Then, the flex divisor satisfies $R_{\rm flex}\in |n_dL|$
where $$n_d = \frac{(2d)!(2d+1)!}{d!^2(d+1)!^2}.$$ \end{theorem}

\begin{proof} By Propositions \ref{int-lemma} and \ref{int-lemma2}, 
we have the formula 
$$n_d=-\sigma_1\cdot c_{2d-1}(\mathcal{O}_G^{[2d]}).$$ From the formula of
Corollary \ref{tot-chern} for
$c(\mathcal{O}_G^{[2d]})$, plus the fact that the minus signs cancel
in any contribution to top degree, we conclude
$$n_d = \left[\sigma_1\cdot
\frac{(1+\sigma_1)^{4d+2}}{(1+\sigma_1+\sigma_2)^{d+2}}\right]_{\rm top}.$$ 
The Pieri and Giambelli formulae imply that
$$\sigma_1^m\cdot \sigma_2^n = \frac{m!}{(m/2)!(m/2+1)!}$$
when $m+2n=2d$ add up to the correct dimension to give a top class on $G$.
After performing binomial expansion in $\sigma_1$ then $\sigma_2$,
collecting terms of top degree,
and plugging in the above formula, we get the ugly expression
$$n_d=
\sum_{j=0}^d\sum_{\ell=1}^{d-j} (-1)^{j+1}{4d+2\choose j} 
{3d-j \choose 2d+\ell}{2d+\ell \choose 2\ell-1}
{2\ell \choose \ell}\frac{1}{\ell+1}.$$ Applying
automated choose identity verification gives the result.
\end{proof}

\begin{example}\label{deg2ex} Let $(X,L)$ be any ADE K3 surface of degree $L^2=2$.
The linear system $|L|$ defines a $2:1$ morphism from $X$ onto either
$\mathbb{P}^2$ or $\mathbb{F}_4^0$ and $R_{\rm flex}$ is naturally
the ramification divisor of this map.
The double cover of $\mathbb P^2$ is branched in a sextic $B$. One has
$R_{\rm flex}^2 = B^2/2 = 18=(3L)^2$, so $n_1=3$.
\end{example}

\begin{example}\label{deg4ex}
For a quartic surface, one can compute the flex divisor directly from
the definition. Here are some results:

The Fermat quartic
$X=V(x_0^4+x_1^4+x_2^4+x_3^4)\subset \mathbb{P}^3$
contains $48$ lines. Each line appears
with multiplicity one in $R_{\rm flex}$. The intersections of $X$ with
the coordinate hyperplanes $x_i=0$ appear with multiplicity $2$ in $R_{\rm flex}$. So
$R_{\rm flex}$ is cut out by $(x_0^4+x_1^4)(x_0^4+x_2^4)(x_0^4+x_3^4)x_0^2x_1^2x_2^2x_3^2.$

 The maximal number of $64$ lines on a smooth quartic surface
 is realized by the Schur quartic $X=V(x_0^4-x_0x_1^3+x_2x_3^3-x_3^4)\subset \mathbb{P}^3$.
These lines come in two types. The first type, of which there are
$16$,
are the lines joining the $4+4$ points lying on the skew lines $V(x_0,x_1)$,
$V(x_2,x_3)$. They
appear
in $R_{\rm flex}$ with multiplicity two, while the remaining $48$
lines of the second type appear with multiplicity one. So $R_{\rm flex}$ consists
only of lines. Thus $X$ has no ``flex points" in the naive sense.
 \end{example}

\begin{remark} Based on the $d=1$ case,
the authors hoped that $R_{\rm flex}$ would be a canonical choice
of polarizing divisor living in a reasonably small multiple of the polarization class,
at least compared to the rational curve divisor $R_{\rm rc}$. But in fact, the formula
of Theorem \ref{final} grows {\it significantly faster} than the
Yau-Zaslow formula, with the switch occurring between
$d=8$ and $d=9$. Asymptotically, $n_d\sim 2^{4d+1}/\pi d^2$
while Yau-Zaslow$_d$ $\sim e^{4\pi \sqrt{d}}/\sqrt{2}d^{27/4}$. \end{remark}

\renewcommand\MR[1]{}
\bibliographystyle{amsalpha} \bibliography{flex}

\end{document}